\documentclass[12pt,a4paper]{amsart}
\usepackage{mathrsfs}
\usepackage{color}
\newtheorem{theo+}              {Theorem}           [section]
\newtheorem{prop+}  [theo+]     {Proposition}
\newtheorem{coro+}  [theo+]     {Corollary}
\newtheorem{lemm+}  [theo+]     {Lemma}
\newtheorem{exam+}  [theo+]     {Example}
\newtheorem{rema+}  [theo+]     {Remark}
\newtheorem{defi+}  [theo+]     {Definition}

\newenvironment{theorem}{\begin{theo+}}{\end{theo+}}

\newenvironment{corollary}{\begin{coro+}}{\end{coro+}}

\usepackage{amsthm}
\theoremstyle{plain} \theoremstyle{remark}
\newtheorem{remark}{Remark}
\newtheorem{example}{Example}

\newtheorem*{ack}{\bf Acknowledgments}

\def \r{\mbox{${\mathbb R}$}}
\def\E{/\kern-1.0em \equiv }

\author{Mehmet Akif AKYOL$^{*}$}
\address{\hskip-\parindent
Department of Mathematics, Faculty of Science and Arts, Bing\"ol University, Bing\"ol 12000, Turkey}
\email{mehmetakifakyol@bingol.edu.tr}

\author{Ye-Lin Ou$^{**}$}

\address{\hskip-\parindent
Department of Mathematics, Texas A $\&$ M University-Commerce,
\newline Commerce,  TX 75429,  USA} 
\email{yelin.ou@tamuc.edu}
\thanks{$^{*}$The first author would like to thank The Scientific and Technological Research
Council of Turkey (T\"UBITAK) for  a scholarship (No.1059B19160050) which allowed him to visit  the Department of Mathematics at Texas A $\&$ M University-Commerce as a Visiting Scholar during the period of 09/2017-05/2018. He is also grateful to Texas A $\&$ M University-Commerce and the Department of Mathematics for the hospitality he received during his visit there where this work was done. 
\newline\indent $^{**}$The second author is supported by a grant from the Simons Foundation ($\#427231$, Ye-Lin Ou)}
\date{05/08/2018}
\begin{document}
\title[Biharmonic Riemannian submersions]{Biharmonic Riemannian submersions}


\maketitle

\section*{Abstract} 
\begin{quote}
{\footnotesize In this paper, we study biharmonic Riemannian submersions. We first derive bitension field of a general Riemannian submersion, we then use it to obtain biharmonic equations for Riemannian submersions with $1$-dimensional fibers and Riemannian submersions with basic mean curvature vector fields of fibers. These are used to construct examples of proper biharmonic Riemannian submersions with $1$-dimensional fibers and to characterize warped products whose projections onto the first factor are biharmonic Riemannian submersions. }

\end{quote}

\maketitle

\section{Introduction}

 A map $\phi:(M, g)\longrightarrow (N, h)$ between
Riemannian manifolds is a {\em biharmonic map} if
$\phi|\Omega$ is a critical point of the bienergy functional
\begin{equation}\nonumber
E_{2}\left(\phi,\Omega \right)= \frac{1}{2} {\int}_{\Omega}
\left|\tau(\phi) \right|^{2}{\rm d}v_g
\end{equation}
for every compact subset $\Omega$ of $M$, where
$\tau(\phi)={\rm Trace}_{g}\nabla {\rm d} \phi$ is the
tension field of $\phi$. Using the first variational formula
(see \cite{Ji}) one obtains the biharmonic map equation
\begin{equation}\label{BI1}
\tau^{2}(\phi):=-\triangle^{\phi}(\tau(\phi)) - {\rm
Trace}_{g} R^{N}({\rm d}\phi, \tau(\phi)){\rm d}\phi
=0,
\end{equation}
where
\begin{equation}\notag
\triangle^{\phi}=-{\rm Trace}_{g}(\nabla^{\phi})^{2}= -{\rm
Trace}_{g}(\nabla^{\phi}\nabla^{\phi}-\nabla^{\phi}_{\nabla^{M}})
\end{equation}
is the Laplacian on sections of the pull-back bundle $\phi^{-1}
TN$ and $R^{N}$ is the curvature operator of $(N, h)$ defined by
$$R^{N}(X,Y)Z=
[\nabla^{N}_{X},\nabla^{N}_{Y}]Z-\nabla^{N}_{[X,Y]}Z.$$

Biharmonic Riemannian submersions were first studied by Oniciuc in \cite{On} where he proved that if a Riemannian submersion $\phi: (M^{m}, g)\longrightarrow (N^n, h)$ has basic tension field, i.e.,  $\tau(\phi)(p)=\tau(\phi)(q)$ provided $\phi(p)=\phi(q)$, (or the tension field is constant along the fibers), then
\begin{align*}
\tau_2(\phi)={\rm Tr}_h(\nabla^N)^2\tau(\phi)+\nabla^N_{\tau(\phi)}\tau(\phi)+{\rm Ric}^N(\tau(\phi)).
\end{align*}
 In particular, it was proved in  \cite{On} that if the tension field $\tau(\phi)$ of a Riemannian submersion $\phi: (M^{m}, g)\longrightarrow (N^n, h)$ is a unitary Killing vector field on $N$, then it is a biharmonic map.\\

The biharmonic map equations for a more general class of maps, i.e., conformal submersions (or even more generally, horizontally weakly conformal maps), were obtained independently in \cite{BFO} and \cite{LO}. As an application of these equations, a family of biharmonic Riemannian submersions were constructed in \cite{LO}. Later in \cite{WO}, biharmonic Riemannian submersions from $3$-dimensional spaces with $1$-dimensional fibers were studied using the so-called integrability data. It was proved in \cite{WO} that  a Riemannian submersion from a $3$-dimensional space form onto a surface is biharmonic if and only if it is harmonic, i.e., the fibers have to be totally geodesic.\\

In a recent paper \cite{GO}, the authors found many examples of biharmonic Riemannian submersions in their study of generalized harmonic morphisms which are maps bewteen Riemannian manifolds that pull back local harmonic functions to local  biharmonic functions. 
\begin{example}(see \cite{GO} for details)
The map $\phi:\r^4\setminus\{(0, 0, 0, x): x\in \r\}=M^4\longrightarrow \r^2$ with $\phi(x_1,\cdots, x_4)=(\sqrt{x_1^2+x_2^2+x_3^2\,}, x_4)$ is a proper biharmonic Riemannian submersion with the mean curvature vector field of fibers being  basic vector field. In fact, using cylindrical coordinates $(r, \theta, \varphi, x_4)$ the Riemannian submersion can be expressed as $\phi (r, \theta, \varphi, x_4)=(r, x_4)$, and the tension field of the Riemannian submersion is $\tau(\phi)={\rm d}\phi(\frac{2}{r}\frac{\partial}{\partial r})$ and the mean curvature vector field of the fibers is $\mu=-\frac{1}{r}\frac{\partial}{\partial r}$.
\end{example}
\begin{example}(see \cite{GO} for details)
 Let $\mathbb{R}^2_{+}=\{ (x,y)\in \r^2: y>0\}$ denote the upper-half plane and $C$ be a positive constant. Then the Riemannian submersion defined by the projection of the warped product
 \begin{align}\notag
\pi  : ( \mathbb{R}^2_{+} \times \mathbb{R} , dx^2 +
dy^2+Cy^4 dz^2) &\to (\mathbb{R}^2_{+} ,dx^2 + dy^2) \\\notag
\pi(x,y,z) =(x,y),
\end{align}
is a proper biharmonic Riemannian submersion with the tension field $\tau(\pi)=\frac{2}{y}{\rm d}\pi(\frac{\partial}{\partial y})$ and the mean curvature vector field of fibers $\mu=-\frac{2}{y}\frac{\partial}{\partial y}$, which is basic.
\end{example}

Note that there are many Riemannian submersions ( e.g., those given by the projections of a twisted product onto its first factor studied in Corollary \ref{Twisted}) whose mean curvature vector fields are not basic.

In this paper, we study biharmonicity of a general  Riemannian submersion. We first derive the bitension field of a general Riemannian submersion, we then use it to obtain biharmonic equations for Riemannian submersions with $1$-dimensional fibers generalizing a result in \cite{WO} and Riemannian submersions with basic mean curvature vector fields of fibers. These are used to construct examples of proper biharmonic Riemannian submersions with $1$-dimensional fibers and to characterize warped products whose projections onto the first factor are biharmonic Riemannian submersions.

\section{Biharmonic equations for Riemannian submersions}
Biharmonic equations for Riemannian submersions can be obtained as a special case of biharmonic equations for conformal submersion obtained in \cite{BFO} and \cite{LO}. However, the equation obtained this way seems to be hard to apply. For the convenience of applications, we first derive the following form of the  bitension field of a Riemannian submersion and then use it to obtain biharmonic equation for Riemannian submersions with $1$-dimensional fibers or with  basic mean curvature vector field of fibers.\\

\begin{theorem}\label{MT}
Let $\phi: (M^m,g)\longrightarrow (N^n,h)$ be a Riemannian submersion between Riemannian manifolds, then the bitension field of $\phi$ is given by
\begin{eqnarray}\notag
\tau_{2}(\phi)&=&-(m-n){\rm d}\phi \Big(  \sum_{i=1}^n\big\{\nabla^{M}_{e_{i}}(\nabla^{M}_{e_{i}}\mu)^{\mathcal{H}}- \nabla^{M}_{(\nabla^{M}_{e_i}e_i)^{\mathcal{H}}}\mu+[\mu, (\nabla^{M}_{e_i}e_i)^{\mathcal{V}}]\big\}\\\notag
&&+\sum_{s=n+1}^m\{ [[\mu,e_s],e_s] +[\mu,(\nabla^{M}_{e_s}e_s)^{\mathcal{V}}]\}-(m-n)\nabla^{M}_{\mu}\mu \Big)\\\label{EQ1}
&&-(m-n){\rm Ricci}^N(\rm d\phi(\mu)),
\end{eqnarray}
where $\{e_{i},e_{s}\}_{i=1,\dots,n, s=1,\dots,m-n}$ is a local orthonormal frame on the total space $M$ with $e_{i}$ horizontal and $e_{s}$ vertical,  $\mu = \tfrac{1}{m-n} \sum_{s=1}^{m-n} (\nabla_{e_{s}} e_{s})^{\mathcal{H}}$
is  the mean curvature vector field of the fibers of the Riemannian submersion.
\end{theorem}
\begin{proof}

It is well known (see e.g., \cite{BW1}, Lemma 4.5.1) that for a Riemannian submersion $\phi: (M^m,g)\longrightarrow (N^n,h)$ and any horizontal vector fields  $X,Y$ and vertical vector fields $V, W$, we have 
\begin{equation}\label{cases1}
\begin{cases}
(\nabla{\rm d }\phi)(X,Y)=0,\\
(\nabla{\rm d }\phi)(V,W)=-\rm d\phi(\nabla^{M}_{V}W),\\
(\nabla{\rm d }\phi)(X,V)=-\rm d\phi(\nabla^{M}_{X}V),\\
\end{cases}
\end{equation}
for all $X, Y \in \Gamma(\mathcal{H})$ and $V, W \in \Gamma(\mathcal{V})$.\\

One can easily check that  the tension field of the Riemannian submersion is given by  
\begin{equation}
\tau(\phi) = -(m-n)d\phi(\mu).
\end{equation}

By the definition of the second fundamental form of the map $\phi$ and (\ref{cases1}) we have
\begin{eqnarray}\notag
\nabla^{\phi}_{e_{i}} d\phi(\mu) &=& d\phi(\nabla^{M}_{e_{i}}\mu)=d\phi\big((\nabla^{M}_{e_{i}}\mu)^{\mathcal{H}}\big),\\ \label{GD1}
\nabla^{\phi}_{e_{i}} \nabla^{\phi}_{e_{i}} d\phi(\mu) &=&\rm d\phi(\nabla^{M}_{e_{i}}(\nabla^{M}_{e_{i}}\mu)^{\mathcal{H}}),\\\label{GD2}
-\nabla^{\phi}_{\nabla^{M}_{e_i}e_i} d\phi(\mu) &=&-\rm d\phi(\nabla^{M}_{(\nabla^{M}_{e_i}e_i)^{\mathcal{H}}}\mu)+d\phi([\mu, (\nabla^{M}_{e_i}e_i)^{\mathcal{V}}]).
\end{eqnarray}
For vertical vector field $e_s$  on $M$, since $\mu$ is horizontal, we have
\begin{align}\label{GD3}
 \nabla^{\phi}_{e_{s}} d\phi(\mu)=-d\phi([\mu,e_s]).
\end{align}
Using (\ref{GD3}) and (\ref{cases1}), we have
\begin{align}\label{GD4}
\nabla^{\phi}_{e_s}\nabla^{\phi}_{e_{s}} d\phi(\mu)&=-\nabla^{\phi}_{e_s}d\phi([\mu,e_s])=d\phi([[\mu,e_s],e_s]).
\end{align}
Using (\ref{cases1}) again, we obtain
\begin{align}\label{GD5}
-\nabla^{\phi}_{\nabla^{M}_{e_s}e_s} d\phi(\mu)&=-\rm d\phi(\nabla^{M}_{(\nabla^{M}_{e_s}e_s)^{\mathcal{H}}}\mu)+d\phi([\mu,(\nabla^{M}_{e_s}e_s)^{\mathcal{V}}]).
\end{align}
Finally, the curvature term in the bitension field formula for a Riemannian submersion becomes
\begin{align}\notag
-\sum_{i=1}^m {\rm R}^{N}(\rm d\phi(e_i), \tau(\phi))\rm d\phi(e_i)&=(m-n)\sum_{i=1}^n R^{N}(\rm d\phi(e_i), \rm d\phi(\mu))\rm d\phi(e_i)\\\label{Ric}
&=-(m-n){\rm Ricci}^N(\rm d\phi(\mu)).   
\end{align}
Substituting (\ref{GD1}), (\ref{GD2}), (\ref{GD4}), (\ref{GD5}), and (\ref{Ric}) into the bitension field formula we have
\begin{eqnarray}\notag
\tau_{2}(\phi)&=&-(m-n){\rm d}\phi \Big(  \sum_{i=1}^n\big\{ \nabla^{M}_{e_{i}}(\nabla^{M}_{e_{i}}\mu)^{\mathcal{H}} 
- \nabla^{M}_{(\nabla^{M}_{e_i}e_i)^{\mathcal{H}}}\mu+[\mu, (\nabla^{M}_{e_i}e_i)^{\mathcal{V}}]\big\}\\\notag
&&+\sum_{s=n+1}^m\{ [[\mu,e_s],e_s] +[\mu,(\nabla^{M}_{e_s}e_s)^{\mathcal{V}}]\}-(m-n)\nabla^{M}_{\mu}\mu \Big)\\\label{EQ1}
&&-(m-n){\rm Ricci}^N(\rm d\phi(\mu)),
\end{eqnarray}
from which the theorem follows.
\end{proof}

\begin{corollary}\label{Co1}
Let $\phi: (M^m,g)\longrightarrow (N^n,h)$ be a Riemannian submersion with basic mean curvature vector field of the fibers. Then it is biharmonic if and only if
\begin{eqnarray}\notag
&&{\rm d}\phi \Big(  \sum_{i=1}^n\big\{\nabla^{M}_{e_{i}}(\nabla^{M}_{e_{i}}\mu)^{\mathcal{H}}- \nabla^{M}_{(\nabla^{M}_{e_i}e_i)^{\mathcal{H}}}\mu\big\}
-(m-n)\nabla^{M}_{\mu}\mu \Big)+{\rm Ricci}^N(\rm d\phi(\mu))=0,
\end{eqnarray}
where $\{e_{i},e_{s}\}_{i=1,\dots,n, s=1,\dots,m-n}$ is a local orthonormal frame on the total space $M$ with $e_{i}$ horizontal and $e_{s}$ vertical,  $\mu = \tfrac{1}{m-n} \sum_{s=1}^{m-n} (\nabla_{e_{s}} e_{s})^{\mathcal{H}}$
is  the mean curvature of the fibers of the Riemannian submersion.
\end{corollary}
\begin{proof}
It is well known (see e.g., \cite{ON}) that for a basic vector field $X$ and a vertical vector field $V$ of a Riemannian submersion, $[X, V]$ is vertical and hence $\rm d \phi([X, V])=0$. From this and Theorem \ref{MT} we obtain the corollary.
\end{proof}

Notice that for any basic vector fields $X, Y$ of a Riemannian submersion $\phi:(M^m,g)\longrightarrow (N^n,h)$, i.e., horizontal vector fields that are $\phi$-related to vector fields $\bar{X}, \bar{Y}$ which means ${\rm d}\phi(X)=\bar{X}$ and ${\rm d}\phi(Y)=\bar{Y}$. In this case, it is well known (see e.g., \cite{ON}) that $\nabla^M_XY$ is basic and $\phi$-related to $\nabla^N_{\bar X}{\bar Y}$ so 
\begin{equation}\label{Ba}
{\rm d}\phi(\nabla^M_XY)=\nabla^N_{\rm d\phi(X)}\rm d\phi(Y).
\end{equation}

Since we can always choose a local orthonormal frame $\{e_i, e_s\}$ adapted to a Riemannian submersion so that $\{e_i: i=1, 2,\dots , n\}$ are basic vector fields, we can use Corollary \ref{Co1} and (\ref{Ba}) to have the following 
\begin{corollary}\label{Co2}
Let $\phi: (M^m,g)\longrightarrow (N^n,h)$ be a Riemannian submersion with basic mean curvature vector field of the fibers. Then it is biharmonic if and only if
\begin{eqnarray}\label{Bas}
{\rm Tr}_h\,(\nabla^{N})^2\tau(\phi)
+\nabla^{N}_{\tau(\phi)}\tau(\phi) +{\rm Ricci}^N(\tau(\phi))=0.
\end{eqnarray}
\end{corollary}

\begin{remark}
One can check that the notion of ``the tension field of a Riemannian submersion being basic" defined in \cite{On} is equivalent to the mean curvature vector field of the fibers of the Riemannian submersion being basic. Thus, our Corollary \ref{Co2} recovers Theorem 4.1 in \cite{On}. 
\end{remark}

\begin{theorem} \label{MT2}
The Riemannian submersion
\begin{align}\notag
\phi: ( M^{m} \times N^n , g_M +e^{2\lambda} g_N) \longrightarrow (M^m , g_M),\;\;
\phi(x,y) =x
\end{align}
defined by the projection of a warped product onto its first factor is biharmonic if and only if
\begin{equation}\label{W-P}
{\rm grad}\, \Delta\lambda+2{\rm Ricci}^{M}({\rm grad}\,\lambda)+\frac{n}{2}{\rm grad}\,|{\rm grad}\lambda|^{2}=0,
\end{equation}
where ${\rm grad}$ and $\Delta$ are the gradient and the Laplace operators on $(M^m, g_M)$. 
\end{theorem}

\begin{proof}

First of all, one can easily see that if $\lambda$ is constant, then the Riemannian submersion $\phi$ is harmonic since it has minimal fibers. In this case, we also have (\ref{W-P}). Now suppose that $\lambda$ is not constant. Let  $P=(M\times N, g_M+e^{2\lambda}g_N)$ be the warped product of $(M^m, g_M)$ and $(N^n, g_N)$ with the warping function $f=e^{\lambda}$. We denote by $\nabla^{P}$, $\nabla^{M}$ and $\nabla^{N}$  the connections of the warped product $(P, g=g_M+e^{2\lambda}g_N)$, the base  $(M^m, g_M)$ and the fiber $(N^n, g_N)$ spaces respectively. Then we have
\begin{equation}\label{Prod}
\begin{cases}
 \nabla^{P}_XY=\nabla^{M}_XY,\;\;\; \nabla^{P}_XV=\nabla^{P}_VX=X(\lambda)V\;\; \\
 \nabla^{P}_VW=-g(V,W){\rm grad}\,\lambda+\nabla^{N}_VW\\
\end{cases}
\end{equation}
where for all $X, Y\in \mathfrak{X}(M)$ and $V, W \in \mathfrak{X}(N)$.\\

Let $\{e_{1},...,e_{m},e_{m+1},...,e_{m+n}\}$ be a local orthonormal frame on the warped product  $P=(M\times N,  g_M +e^{2\lambda} g_N)$ such that $\{e_{1},...,e_{m}\}$ are horizontal 
and $\{e_{m+1},...,e_{m+n}\}$ are vertical. Clearly this orthonormal frame is adapted to the Riemannian submersion $\phi$. Using (\ref{Prod}) we obtain the tension field of the Riemannian submersion 
\begin{eqnarray}\label{GD40}
\tau(\phi) =-n{\rm d}\phi (\mu)=-\sum_{s=m+1}^{m+n}{\rm d} \phi(\nabla^{P}_{e_{s}}e_{s})=n\, {\rm d} \phi({\rm grad}\lambda)=n\, {\rm grad}\lambda,
\end{eqnarray}
where the last equality was obtained by using the fact that $\phi$ is a projection and hence ${\rm d} \phi(X)=X$ for any horizontal vector field $X$.\\

Clearly, in this case, the mean curvature vector field $\mu=-{\rm grad} \lambda$ of the fibers of the Riemannian submersion is basic. So, by Corollary \ref{Co2} and (\ref{GD40}), the Riemannian submersion $\phi$ is biharmonic if and only if 

\begin{eqnarray}\label{GD41}
{\rm Tr}_{g_M}\,(\nabla^{M})^2 {\rm grad}\lambda
+n \nabla^{M}_{{\rm grad}\lambda}{\rm grad}\lambda +{\rm Ricci}^M({\rm grad}\lambda)=0.
\end{eqnarray}

We can check (see e.g., \cite{BFO}) that
\begin{equation}\label{GD42}
{\rm Tr}_{g_M} (\nabla^M)^{2}({\rm grad}\,\lambda)={\rm grad}\, \Delta\lambda+{\rm Ricci}^{M}({\rm grad}\lambda).
\end{equation}

On the other hand, it is  the well known  (see e.g., \cite{Pe}) that
\begin{equation}\label{GD43}
\nabla^{M}_{{\rm grad}\lambda}{\rm grad}\lambda=\frac{1}{2}{\rm grad} (|{\rm grad}\lambda|^{2}).
\end{equation}

Substituting (\ref{GD42}) and (\ref{GD43}) into  (\ref{GD41}) we obtain Equation (\ref{W-P}), which completes the proof of the theorem.
\end{proof}
\begin{remark}
 It is interesting to note that in their study of biharmonic maps from warped product manifolds, Balmus-Montaldo-Oniciuc \cite{BMO} introduced the notion of {\em biharmonic warping function} defined to be a warping function $f$ so that the identity map ${\rm Id}:M\times_{f}N\equiv (M\times N, g_M+f^2g_N)\longrightarrow (M\times N, g_M+g_N)$, ${\rm Id}(x,y)=(x,y)$ is a proper biharmonic map for some manifold $N$ of dimension greater than zero. They also proved that the projection
\begin{align}\notag
\phi: ( M^{m} \times N^n , g_M +f^2 g_N) \longrightarrow (M^m , g_M),\;\;
\phi(x,y) =x
\end{align}
 is  proper biharmonic if and only if
\begin{equation}
{\rm Tr}_g\nabla ^2 {\rm grad}\, \ln \,f+{\rm Ricci}^{M}({\rm grad}\, \ln \,f)+\frac{n}{2}{\rm grad}\,|{\rm grad}\, \ln \,f|^2=0.
\end{equation}
(ii)  Note also that the biharmonic equation for the projection of a warped product in the form of (\ref{Bas}) was obtained in \cite{Ur1} as in this case the mean curvature vector field of the Riemannian submersion is basic. So our Theorem \ref{MT2} gives an improvement of the equation obtained in \cite{Ur1}.\\
\end{remark}

\begin{corollary}
If $ (M^m, g_M)$ is an Einstein Manifold with ${\rm Ricci}^{M}=ag_M$, then the Riemannian submersion
\begin{align}\notag
\phi: ( M^{m} \times N^n , g_M +e^{2\lambda} g_N) \longrightarrow (M^m , g_M),\;\;
\phi(x,y) =x
\end{align}
defined by the projection of a warped product onto its first factor is biharmonic if and only if
\begin{equation}\label{Einstein}
 \Delta\lambda+2a\lambda+\frac{n}{2}|{\rm grad}\lambda|^{2}=C,
\end{equation}
where $C$ is a constant and $\Delta$ is the Laplace operator on $(M^m, g_M)$. 
\end{corollary}
\begin{proof}
If $ (M^m, g_M)$ is an Einstein Manifold with ${\rm Ricci}^{M}=a\,g_M$, then one can easily check that ${\rm Ricci}^{M}({\rm grad}\lambda)=a\, {\rm grad}\lambda$. Substituting this into (\ref{W-P}) we conclude that the Riemannian submersion
\begin{align}\notag
\phi: ( M^{m} \times N^n , g_M +e^{2\lambda} g_N) \longrightarrow (M^m , g_M),\;\;
\phi(x,y) =x
\end{align}
defined by the projection of a warped product onto its first factor is biharmonic if and only if
\begin{equation}\notag
 {\rm grad}\Big(\Delta\lambda+2a\lambda+\frac{n}{2}|{\rm grad}\lambda|^{2}\Big)=0,
\end{equation}
from which we obtain the corollary.
\end{proof}

\begin{corollary}\label{MT3}
 The Riemannian submersion
\begin{align}\notag
\phi: ( M^{m} \times N^n , g_M +e^{2\lambda} g_N) \longrightarrow (M^m , g_M),\;\;
\phi(x,y) =x
\end{align}
defined by the projection of a warped product onto a compact manifold with negative Ricci curvature is  biharmonic if and only if
it is harmonic.
\end{corollary}
\begin{proof}
By Theorem \ref{MT2}, the Riemannian submersion
\begin{align}\notag
\phi: ( M^{m} \times N^n , g_M +e^{2\lambda} g_N) \longrightarrow (M^m , g_M),\;\;
\phi(x,y) =x
\end{align}
defined by the projection of a warped product onto its first factor is biharmonic if and only if
\begin{equation}\label{W1}
{\rm grad}\, \Delta\lambda+2{\rm Ricci}^{M}({\rm grad}\,\lambda)+\frac{n}{2}{\rm grad}\,|{\rm grad}\lambda|^{2}=0.
\end{equation}
On the other hand, it is well known that for any function $\lambda$ on $M$, we have the following Bochner formula
\begin{equation}\notag
\frac{1}{2} \Delta|{\rm grad}\, \lambda|^2=|\nabla{\rm d}\, \lambda|^2+ \langle {\rm grad}\, \lambda, {\rm grad}\, \Delta \lambda\rangle +{\rm Ricci}^{M}({\rm grad}\,\lambda, {\rm grad}\,\lambda).
\end{equation}
Substituting the term ${\rm grad}\, \Delta\lambda$ from (\ref{W1}) into the Bochner formula we have
\begin{equation}\label{W2}
\frac{1}{2} \Delta|{\rm grad}\, \lambda|^2=|\nabla{\rm d}\, \lambda|^2 - \frac{n}{2} \langle {\rm grad}\, \lambda, {\rm grad}\,|{\rm grad}\lambda|^{2}\rangle -{\rm Ricci}^{M}({\rm grad}\,\lambda, {\rm grad}\,\lambda).
\end{equation}
If $\lambda$ is constant, then the Riemannian submersion is harmonic and the corollary follows. Now suppose $\lambda$ is not a constant, then, since $M$ is compact the function $|\nabla{\rm d}\, \lambda|^2$ attains its maximum at some point $p\in M$. Evaluating both sides of Equation (\ref{W2}) at the point $p$ we see that the right-hand side is positive  by the fact that $({\rm grad}\,|{\rm grad}\lambda|^{2})(p)=0$ and the assumption of negative Ricci curvature whilst the left-hand side is non-positive since the Laplacian being the trace of the Hessian of $|{\rm grad}\, \lambda|^2$, which is semi-negative definite at $p$. The contradiction shows that $\lambda$ has to be a constant on $M$. This completes the proof of the corollary.

\end{proof}

\begin{example}
For any constants $a_i>0, \,b_i$,  $i=1,2,...,n,$  let $\kappa_i(x_i)=-\frac{2}{x_i+b_i}$,  $\kappa_i(x_i)=a_i\tan\big(\frac{a_i}{2}(x_i+b_i)\big),$ or
$\kappa_i(x_i)=\frac{a_i(1+e^{a_i(x_i+b_i)})}{1-e^{a_i(x_i+b_i)}},$ and let $\lambda(x_1,x_2,...,x_n)=- \sum^{n}_{i=1} \int\kappa_i(x_i)dx_i,$ then the Riemannian submersion
\begin{align}\notag
\pi  :\big(  \mathbb{R}^{n} \times \mathbb{R} , g_0 +e^{2\lambda(x)} dt^2\big) \longrightarrow (\mathbb{R}^n ,g_0),\;\;\pi(x,t) =x
\end{align}
given by the projection of a warped product onto its first factor is a proper biharmonic map. In particular, the projections
 \begin{align}\notag
\pi  :\big( (0,\infty)^n \times \mathbb{R} , g_0 +(x_1 x_2\cdots x_n)^4 dt^2\big) \longrightarrow ((0,\infty)^n ,g_0),\;\;\pi(x,t) =x,
\end{align}
and 
 \begin{align}\notag
\pi  :\big( (0,\infty)^n \times \mathbb{R} , g_0 +(\sinh x_1 \sinh x_2\cdots \sinh x_n)^2 dt^2\big) \longrightarrow ((0,\infty)^n ,g_0),\;\;\pi(x,t) =x,
\end{align}
are both biharmonic Riemannian submersions.

In fact, since the base manifold is an open submanifold of a Euclidean space which is an Einstein space with $a=0$, we use this and Equation (\ref{Einstein}) with $n=1$ to conclude that the Riemannian submersion $\pi$ is biharmonic if and only if 
\begin{equation}\label{EinP1}
 \Delta\lambda+\frac{1}{2}|{\rm grad}\lambda|^{2}=C,
\end{equation}
where the Laplacian and the gradient operators are defined by the Euclidean metric. If we look for the special solutions of the form $\lambda=-\sum^{n}_{i=1} \int\kappa_i(x_i)dx_i$, then a straightforward computation shows that Equation (\ref{EinP1}) is equivalent to 
\begin{align}\label{P20}
 -\sum_{i=1}^n\big( \kappa^{\prime}_i-\frac{\kappa^{2}_i}{2}\big)=C.
\end{align}
Since $\kappa_i=\kappa_i(x_i)$, one can easily check that (\ref{P20}) is equivalent to

\begin{align}\label{P2}
 \kappa^{\prime}_i-\frac{\kappa^{2}_i}{2}=C_i,  \ i=1,2,...,n,
\end{align}
which can be solved by the following cases.\\ 

\textbf{Case I.} For $C_i=0$,  we solve (\ref{P2}) to have solutions
$$
\kappa_i=-\frac{2}{x_i+b_i},\ \ \ i=1,2,...,n,
$$
for some constant $b_i.$

\textbf{Case II.} For  $C_i=\frac{a_i^2}{2}>0$ with constant $a_i>0$, then it is easily chcked that (\ref{P2}) can be solved to have
$$
\kappa_i=a_i\tan\Big(\frac{a_i}{2}(x_i+b_i)\Big),\ \ \ i=1,2,...,n,
$$
for some constant $b_i.$

\textbf{Case III.} For $C_i=-\frac{a_i^2}{2}<0$, we solve (\ref{P2}) to obtain

$$
\kappa_i=\frac{a_i(1+e^{a_i(x_i+b_i)})}{(1-e^{a_i(x_i+b_i)})},\ \ \ i=1,2,...,n,
$$
for some constant $b_i.$\\

From these, we obtain the example.
\end{example}

\section{Biharmonic Riemannian submersions with $1$-dimensional fibers}

 In this section, we will follow the idea from \cite{WO} to describe biharmonicity of a Riemannian submersion with $1$-dimensional fibers
from a generic $(n+1)$-manifold by using the integrability data of a
special orthonormal frame adapted to a Riemannian submersion.  We also use the biharmonic equation derived to study biharmonicity of a Riemannian submersion defined by the projection of a twisted product onto it first factor.\\

Let $\phi:( M^{n+1} , g)\longrightarrow (N^n, h)$ be a Riemannian
submersion. A local orthonormal frame is said to be {\bf adapted to
the Riemannian submersion} $\phi$ if the vector fields in the frame
that are tangent to the horizontal distribution are basic (i.e.,
they are $\phi$-related to a local orthonormal frame in the base
space). Such a frame always exists (cf. e.g., \cite{BW1}). Let
$\{e_1,\ldots,e_n, e_{n+1}\}$ be an orthonormal frame adapted
to $\phi$ with $e_{n+1}$ being vertical. Then, it is well known (see
\cite{ON}) that $[e_i,e_{n+1}]$ $(i=1, 2, \ldots, n)$ are vertical
and $[e_i,e_j]$ $(i, j=1, 2, \ldots, n)$ are $\phi$-related to
$[\varepsilon_i, \varepsilon_j]$, where $\{\varepsilon_i, \ldots,
\varepsilon_n\}$ is an orthonormal frame in the base manifold. If we
assume that
\begin{equation}\label{RS22}
[\varepsilon_i,\varepsilon_j]=F_{ij}^k\varepsilon_k,
\end{equation}
for $F^k_{ij}\in C^{\infty}(N)$ and use the notations
$f^k_{ij}=F^k_{ij}\circ \phi,\;\forall\; i, j, k=1, 2, \ldots, n$. Then, we have
\begin{equation}\label{LB}
\begin{cases}
[e_i,e_{n+1}]=\kappa_ie_{n+1},\\
[e_i,e_j]=f^k_{ij} e_k-2\sigma_{ij}e_{n+1},\;\;\; i, j=1, 2, \ldots,
n,
\end{cases}
\end{equation}
where $\kappa_i,\,{\rm and}\; \sigma_{ij} \in C^{\infty}(M)$. We
will call $ \{f^k_{ij},  \kappa_i,\,\sigma_{ij}\}$ {\bf the
integrability data} of the adapted frame of the Riemannian
submersion $\phi$. It follows from (\ref{LB}) that 
\begin{eqnarray}
f_{ij}^k=-f_{ji}^k\; \;{\rm and}\;\sigma_{ij}=-\sigma_{ji},\;\;\forall\; i,j, k=1,2, \cdots, n.
\end{eqnarray}

\begin{theorem}\label{1D}
 Let $\phi:(M^{n+1},g)\longrightarrow (N^n,h)$ be a Riemannian
submersion with the adapted frame $\{e_1,\ldots, e_{n+1}\}$ and the
integrability data $ \{f^k_{ij}, \kappa_i, \sigma_{ij}\}$. Then,
the Riemannian submersion $\phi$ is biharmonic if and only if
\begin{eqnarray}\label{1DE}
&&\Delta\kappa_k+\sum^{n}_{i,j=1}\Big(2e_i(\kappa_j)P^{k}_{ij}+\kappa_j\big[ e_i (P^{k}_{ij})+P^{l}_{ij}P_{il}^k -\kappa_iP_{ij}^k - P_{ii}^lP_{lj}^k\big] \Big)\\\notag
&&+{\rm Ricci}^N(\rm d\phi(\mu), {\rm d}\phi(e_k))=0, \;\;k=1,2,\ldots, n,
\end{eqnarray}
where $P_{ij}^k=\frac{1}{2}(-f_{ik}^j-f_{jk}^i +f_{ij}^k)$ for all $ i,j, k=1,2, \cdots, n$ and $\Delta$ denotes the Laplacian of $(M^{n+1}, g)$.
\end{theorem}
\begin{proof}
Let $P_{ij}^k=\frac{1}{2}(-f_{ik}^j-f_{jk}^i +f_{ij}^k)$ for all $ i,j, k=1,2, \cdots, n$. Then, one can easily check that
\begin{eqnarray}
&&P_{ij}^k =P_{ji}^k+f_{ij}^k, \;\; P_{ii}^k=-f_{ik}^i.
\end{eqnarray}
A straightforward computation using (\ref{LB})
and Koszul formula gives 
\begin{eqnarray}\label{NAi}
\nabla_{e_{i}} e_{j} &=& P_{ij}^ke_k-\sigma_{ij}e_{n+1}\;\;\;{\rm  for\; any\; } i, j=1,2, \cdots, n,\\\label{NAs}
\nabla_{e_{n+1}} e_{n+1} &=& \sum_{i=1}^n\kappa_ie_i.
\end{eqnarray}

The mean curvature vector field of the fibers of the Riemannian  submersion $\phi$ is given by
\begin{eqnarray}\label{MC}
\mu=(\nabla_{e_{n+1}}e_{n+1})^{\mathcal{H}}=\sum_{i=1}^n\kappa_ie_i.
\end{eqnarray}
A straightforward computation using (\ref{LB}), (\ref{NAi}), (\ref{NAs}) and (\ref{MC}) yields
\begin{eqnarray}\notag
 && \sum_{i=1}^n\big\{\nabla^{M}_{e_{i}}(\nabla^{M}_{e_{i}}\mu)^{\mathcal{H}}\big\}^{\mathcal{H}}=\big\{\sum^{n}_{i,j=1}\nabla^{M}_{e_i}\Big(e_i(\kappa_j)+\sum^{n}_{l=1}\kappa_lP^{j}_{il}\Big)e_j\big\}^{\mathcal{H}}\\\notag
&&=\sum^{n}_{i,k=1}e_i\Big(e_i(\kappa_k)+\sum^{n}_{j=1}\kappa_jP^{k}_{ij}\Big)e_k+\sum^{n}_{i,j=1}\Big(e_i(\kappa_j)+\sum^{n}_{l=1}\kappa_lP^{j}_{il}\Big)P_{ij}^ke_k\\\label{GD20}
&&=\sum^{n}_{k=1}\Big(\sum^{n}_{i=1}e_ie_i(\kappa_k)+2\sum^{n}_{i,j=1}e_i(\kappa_j)P^{k}_{ij}+\sum^{n}_{i,j=1}\kappa_j(e_i P^{k}_{ij})+\sum^{n}_{i,j=1}\kappa_jP^{l}_{ij}P_{il}^k\Big)e_k\\
&&  \big\{- \nabla^{M}_{(\nabla^{M}_{e_i}e_i)^{\mathcal{H}}}\mu\big\}^{\mathcal{H}}=-\sum_{i,k,j=1}^n\big( P_{ii}^j(e_j\kappa_k)+\kappa_j P_{ii}^lP_{lj}^k  \big)e_k\\
&&  [\mu, (\nabla^{M}_{e_i}e_i)^{\mathcal{V}}]=0,\; {\rm since}\; (\nabla^{M}_{e_{i}}e_{i})^{\mathcal{V}}=0,\\
&&  [[\mu,e_{n+1}],e_{n+1}] =\sum_{k=1}^n e_{n+1}(e_{n+1}\,\kappa_k)e_k\\
&&  [\mu,(\nabla^{M}_{e_{n+1}}e_{n+1})^{\mathcal{V}}]=0,\; {\rm since}\; (\nabla^{M}_{e_{n+1}}e_{n+1})^{\mathcal{V}}=0,\\
&& \big\{-\nabla^{M}_{\mu}\mu \big\}^{\mathcal{H}}=-\sum_{i,k=1}^n \kappa_i\Big(e_i(\kappa_k)+\sum_{j=1}^n \kappa_jP_{ij}^k )\Big)e_k\\
&& {\rm Ricci}^N(\rm d\phi(\mu))=\sum_{k=1}^n{\rm Ricci}^N(\rm d\phi(\mu), {\rm d}\phi(e_k)){\rm d}\phi(e_k).
\end{eqnarray}

Using these and Theorem \ref{MT} we conclude that the Riemannian submersion $\phi$ is biharmonic if and only if

\begin{eqnarray}\notag
&&\sum^{n}_{k=1}\Big(\sum^{n}_{i=1}e_ie_i(\kappa_k)+2\sum^{n}_{i,j=1}e_i(\kappa_j)P^{k}_{ij}+\sum^{n}_{i,j=1}\kappa_j(e_i P^{k}_{ij})+\sum^{n}_{i,j=1}\kappa_jP^{l}_{ij}P_{il}^k\\\notag
&&  -\sum_{i,j=1}^n\big( P_{ii}^j(e_j\kappa_k)+\kappa_j P_{ii}^lP_{lj}^k \big) + e_{n+1}(e_{n+1}\,\kappa_k)-\sum_{i=1}^n \kappa_ie_i(\kappa_k)-\sum_{i,j=1}^n \kappa_i\kappa_jP_{ij}^k \\\notag
&&+{\rm Ricci}^N\big(\rm d\phi(\mu), {\rm d}\phi(e_k)\big)\Big){\rm d}\phi(e_k)=0.
\end{eqnarray}
This is equivalent to 
\begin{eqnarray}\notag
&&\sum^{n+1}_{i=1}e_ie_i(\kappa_k) -\sum_{i=1}^n P_{ii}^j(e_j\kappa_k) -\sum_{i=1}^n \kappa_ie_i(\kappa_k)\\\label{GD30}
&&+\sum^{n}_{i,j=1}\Big(2e_i(\kappa_j)P^{k}_{ij}+\kappa_j(e_i P^{k}_{ij})+\kappa_jP^{l}_{ij}P_{il}^k -\kappa_jP_{ij}^k - \kappa_j P_{ii}^lP_{lj}^k\Big)\\\notag
&&+{\rm Ricci}^N(\rm d\phi(\mu), {\rm d}\phi(e_k))=0, \;\;k=1,2,\ldots, n
\end{eqnarray}
since $\{{\rm d}\phi(e_k)|\;k=1,2,\ldots, n\}$ is an orthonormal frame on $\phi^{-1}TN$.
Finally, a straightforward computation of the Laplacian $\Delta\kappa_k$ shows that Equation (\ref{GD30}) is equivalent to the following
\begin{eqnarray}\notag
&&\Delta\kappa_k+\sum^{n}_{i,j=1}\Big(2e_i(\kappa_j)P^{k}_{ij}+\kappa_j\big[ e_i (P^{k}_{ij})+P^{l}_{ij}P_{il}^k -\kappa_iP_{ij}^k - P_{ii}^lP_{lj}^k\big] \Big)\\\notag
&&+{\rm Ricci}^N(\rm d\phi(\mu), {\rm d}\phi(e_k))=0, \;\;k=1,2,\ldots, n,
\end{eqnarray}
from which we obtain the theorem.
\end{proof}

\begin{remark}
One can check that when $n=2$,  the integrability data is given by
\begin{eqnarray}\notag
f_{12}^1&=& f_1,\; f_{12}^2=f_2,\\\notag
P^{2}_{11}&=&-f_1, \; P^{2}_{21}=-f_2, \; P^{1}_{12}=f_1, \; P^{1}_{22}=f_2,\; {\rm all\;other}\; P_{ij}^k=0.
\end{eqnarray}
Substituting this and ${\rm Ricci}^N(\rm d\phi(\mu), {\rm d}\phi(e_k))=\kappa_k{\rm K}^N$ into Equation (\ref{1DE}), we obtain that a Riemannian
submersion $\phi:(M^{3},g)\longrightarrow (N^2,h)$  is biharmonic if and only if
\begin{eqnarray}
&&\Delta\kappa_1+(e_1\kappa_2)f_1+(e_2\kappa_2)f_2+e_1(\kappa_2  f_1)+e_2(\kappa_2 f_2)\\\notag
&& -\kappa_1\kappa_2f_1-(\kappa_2)^2f_2 -\kappa_1(-{\rm K}^N+f_1^2+ f_2^2)=0,\\
&&\Delta\kappa_1-(e_1\kappa_1)f_1-(e_2\kappa_1)f_2-e_1(\kappa_1  f_1)-e_2(\kappa_1 f_2)\\\notag
&& +\kappa_1\kappa_2f_2+(\kappa_1)^2f_1 -\kappa_2(-{\rm K}^N+f_1^2+ f_2^2)=0.
\end{eqnarray}
This is exactly Proposition 2.1 in \cite{WO} of which our Theorem \ref{1D} gives a generalization.
\end{remark}

\begin{corollary}\label{Twisted}
 Let $g_0$ be the standard Euclidean metric on $\r^n$. Then, the Riemannian submersion defined by the projection of the twisted product
\begin{align}\notag
\phi  : ( \mathbb{R}^{n} \times \mathbb{R} , g_0 +e^{2\lambda(x,t)} dt^2) &\to (\mathbb{R}^n ,g_0),\;\; \phi(x,t) =x
\end{align}
 is biharmonic if and only if its tension field viewed as a map 
 \begin{equation}
 \tau(\phi):( \r^n \times \mathbb{R} , g_0+e^{2\lambda(x,t)}  dt^2) \longrightarrow\r^n, \;\tau(\phi)(x,t)=(\kappa_1(x,t),\kappa_2(x,t),\ldots, \kappa_n(x,t))
 \end{equation}
  is harmonic, where $\tau(\phi)=\sum^n_{i=1}\kappa_i(x,t){\rm d}\phi(e_i)$ is the tension field of the Riemannian submersion with respect to an adapted frame.
 \end{corollary}
\begin{proof}
In this case, we can check that the orthonormal frame  
$\{e_i=\frac{\partial}{\partial x_i}|\, i=1,\ldots, n,\,e_{n+1}= e^{-\lambda}\frac{\partial}{\partial t}\} $  is adapted to
the Riemannian submersion $\phi$ with $ d\phi(e_i)=\varepsilon_i, i=1,
2,...,n$ and $e_{n+1}$ being vertical, where
$\varepsilon_i=\frac{\partial}{\partial x_i}, i=1,2,...,n$ form an orthonormal frame on the
base space $(\mathbb{R}^n ,g_0)$.  An easy  computation gives
 \begin{align}\notag
 [e_i,e_j]=0,\ \  i,j=1,2,...,n,\;\; [e_i,e_{n+1}]=\kappa_ie_{n+1},
\end{align}
where $\kappa_i=-\lambda_{x_i}$. It follows that the integrability data in this case are given by  
\begin{align*}
f_{ij}^k=\sigma_{ij}=0,\;\; \forall\; i,j,k=1,2,\ldots, n.
\end{align*}
It follows that
\begin{equation}\label{111}
\begin{array}{lll}
P_{ij}^{k}=\sigma_{ij}=0 ,\;\; \forall\; i,j,k=1,2,\ldots, n.
\end{array}
\end{equation}

Using these and Theorem \ref{1D} we conclude that the Riemannian submersion defined by the twisted product onto its first factor is biharmonic  if and only if
\begin{equation}\label{Delta}
\Delta \kappa_i=0,\forall\; i=1,2,\ldots, n.
\end{equation}
This means exactly that the tension field of the Riemannian submersion viewed as a map 
 \begin{equation}
 \tau(\phi):( \r^n \times \mathbb{R} , g_0+e^{2\lambda(x,t)}  dt^2) \longrightarrow\r^n, \;\tau(\phi)(x,t)=(\kappa_1,\kappa_2,\ldots, \kappa_n)
 \end{equation}
  is harmonic.
 \end{proof}

\begin{remark}
(i) We would like to point out that the mean curvature vector field of the Riemannian submersion studied in Corollary \ref{Twisted}  is $\mu=-\sum_{i=1}^n\frac{\partial \lambda(x, t)}{\partial x_i}e_i$, which is not  a basic vector field.\\
(ii) One can also check that
\begin{align}\label{correction}
\Delta_{M}\kappa_i&=\sum^{n}_{j=1}\frac{\partial^{2} \kappa_i}{\partial {x_j}^{2}}+e^{-2\lambda}\frac{\partial^2\kappa_i}{\partial {t^2}}-\sum^{n}_{j=1}\kappa_j\frac{\partial\kappa_i}{\partial x_j}-e^{-2\lambda}\frac{\partial \lambda}{\partial t}\frac{\partial \kappa_i}{\partial t}.
\end{align}
\end{remark}

\begin{ack}
Both authors would like to thank C. Oniciuc for some useful comments and questions that help to clarify some statements of  the original manuscript. The second author also wants to thank Changyou Wang for some constructive discussions and sharing his insight on Equation (\ref{Einstein}) in the paper and Fred Wilhelm for some useful communication on constructions of Riemannian submersions.
\end{ack}


\begin{thebibliography}{99}
\bibitem{BFO} P. Baird, A. Fardoun and S. Ouakkas, {\em Conformal and semi-conformal biharmonic maps}, Ann. Glob. Anal. Geom., 34
(2008), no. 4, 403--414.
\bibitem{BW1} P. Baird and  J. C. Wood, {\em Harmonic morphisms between Riemannian manifolds}, London Math. Soc. Monogr.
(N.S.) No. 29, Oxford Univ. Press (2003).
\bibitem{BMO} A. Balmus, S. Montaldo, C. Oniciuc, {\em Biharmonic maps between warped product manifolds}, J. Geom. Phys. 57 (2007), no. 2, 449--466.
\bibitem{GO} E. Ghandour and Y. -L. Ou, {\em Generalized harmonic morphisms and horizontally weakly conformal biharmonic maps}, arXiv:1712.03593, to appear in J. Math, Anal. Appl., 2018.
\bibitem{Ji} G. Y. Jiang, {\em $2$-Harmonic maps and their first and second variational formulas}, Chin. Ann. Math. Ser. A 7(1986) 389-402.
\bibitem{LO} E. Loubeau and Y. -L. Ou, {\em Biharmonic maps and morphisms from conformal mappings}, Tohoku Math J., 62 (1), (2010), 55-73.
\bibitem{ON} B. O'Neill, {\em The fundamental equations of a submersion}, Michigan Math. J. 13 (1966), 459--469.
\bibitem{On} C. Oniciuc, {\em Biharmonic maps between Riemannian manifolds},  An. Stiint. Univ. Al. I. Cuza Iasi. Mat. (N.S.)  48  (2002),  no. 2, 237-248 (2003). 
\bibitem{Pe} P. Petersen, {\em Riemannian geometry}. Graduate Texts in Mathematics, 171. Springer-Verlag, New York, 1998.
\bibitem{Ur1} H. Urakawa, {\em Harmonic maps and biharmonic maps on principal bundles and warped products},  J. Korean Math. Soc., 55(3), (2018), 553-574.
\bibitem{Ur2} H. Urakawa, {\em Harmonic maps and biharmonic Riemannian submersions}, Preprint, 2018.
\bibitem{WO} Z. -P. Wang and   Y. -L. Ou,  {\em Biharmonic Riemannian submersions from 3-manifolds},  Math.  Zeitschrift, 269 (3) (2011), 917-925.

\end{thebibliography}
\end{document}